\documentclass[12pt,reqno]{amsart}
\usepackage{amsmath, amsthm, amssymb, stmaryrd}
\usepackage{hyperref}
\usepackage{enumerate}
\usepackage{url}

\let\OLDthebibliography\thebibliography
\renewcommand\thebibliography[1]{
  \OLDthebibliography{#1}
  \setlength{\parskip}{0pt}
  \setlength{\itemsep}{0pt plus 0.3ex}
}

\usepackage{mathtools}

\usepackage{multirow, array}
\usepackage{placeins}

\usepackage{caption} 
\advance \topmargin by -\headheight
\advance \topmargin by -\headsep
     
\evensidemargin \oddsidemargin
\newtheorem{thm}{Theorem}[section]
\newtheorem{lemma}[thm]{Lemma}

\newtheorem{cor}[thm]{Corollary}

\theoremstyle{definition}

\theoremstyle{remark}

\numberwithin{equation}{section}

\newcommand{\mmod}[1]{{\,\,\mathrm{mod}\,\,#1}}

\newcommand*\wrapletters[1]{\wr@pletters#1\@nil}
\def\wr@pletters#1#2\@nil{#1\allowbreak\if&#2&\else\wr@pletters#2\@nil\fi}

\def\alp{{\alpha}} 
\def\bet{{\beta}}  
\def\gam{{\gamma}} 
\def\del{{\delta}}

\def\tet{{\theta}}  

 \def\Lam{{\Lambda}}

\def\sig{{\sigma}}

\def\eps{\varepsilon}

\def\le{\leqslant} \def\ge{\geqslant}

\def\d{{\,{\rm d}}}

\def \sig{{\sigma}}

\def \bC {\mathbb C}

\def \bN {\mathbb N}

\def \bQ {\mathbb Q}
\def \bR {\mathbb R}
\def \bZ {\mathbb Z}
\def \bT {\mathbb T}

\def \bc {\mathbf c}
\def \bd {\mathbf d}

\def \bn {\mathbf n}

\def \bx {\mathbf x}

\def \by {\mathbf y}
\def \bz {\mathbf z}

\def \bzero {\mathbf 0}

\def \fm {\mathfrak m}
\def \fn {\mathfrak n}

\def \fM {\mathfrak M}

\def \cA {\mathcal A}

\def \cM {\mathcal M}
\def \cN {\mathcal N}

\def \cP {\mathcal P}

\def \cR {\mathcal R}

\def \ord {\mathrm{ord}}

\def \meas {\mathrm{meas}}

\def \diam {\diamond}
\def \Li {{\mathrm{Li}}}
\def \supp {{\mathrm{supp}}}

\begin{document}
\title[Roth--Waring--Goldbach]{Roth--Waring--Goldbach}
\author[Sam Chow]{Sam Chow}
\address{School of Mathematics, University of Bristol, University Walk, Clifton, Bristol BS8 1TW, United Kingdom}
\email{Sam.Chow@bristol.ac.uk}
\thanks{}
\date{}
\begin{abstract} We use Green's transference principle to show that any subset of the $d$th powers of primes with positive relative density contains nontrivial solutions to a translation-invariant linear equation in $d^2+1$ or more variables, with explicit quantitative bounds.
\end{abstract}
\maketitle

\section{Introduction}
\label{intro}

Waring's problem \cite{War1770} dates back to 1770, and asks how large $s$ has to be in terms of $d$ to ensure that if $n$ is a large positive integer then 
\begin{equation} \label{WaringEq}
x_1^d + \ldots + x_s^d = n
\end{equation}
has a solution $\bx \in \bN^s$. The Hardy--Littlewood circle method has been a particularly effective approach to such problems, with the best results due to Wooley --- see \cite{VW2002, Woo2015}, as well as \cite[\S 2]{Bou2016} and \cite{Woo2016}. The circle method has also been used to solve the ternary Goldbach problem, and other problems concerning the addition of primes \cite{Hel2014, Vin1937}. Since Hua \cite{Hua1939}, many authors have enjoyed working on the Waring--Goldbach problem, which considers prime solutions to \eqref{WaringEq} --- see \cite{KW2001, KW2015,  LZ2015, Tha1987, Tha1989}, for instance. The circle method has again been the weapon of choice, with the main technical issue being the study of exponential sums over primes \cite{Hua1965}.

Roth's theorem \cite{Rot1953} states that if $A \subset [N]$ contains no nontrivial three-term arithmetic progressions then $|A| \ll \frac N {\log \log N}$. This bound has since been improved, most recently by Bloom \cite{Blo2015}. Such results are interesting because they identify patterns in the set $A$ without assuming anything about its structure. Three-term arithmetic progressions pertain to the diophantine equation
\[
x - 2y + z = 0,
\]
and much of the arithmetic combinatorics literature surrounds linear equations. Smith \cite{Smi2009}, Keil \cite{Kei2014} and Henriot \cite{Hen2015, Hen2016} have considered higher degree systems with the property that the solution set is invariant under translations and dilations. This property allows the use of a density increment strategy, which is the standard approach to Roth's theorem.

In 2005, Green \cite{Gre2005} famously solved a problem of Roth--Goldbach type for three primes. He devised a means of transferring Roth-type results from the integers to the primes. Using this mechanism, he showed that any subset of the primes with positive \emph{relative density} contains nontrivial three-term arithmetic progressions. Recently Browning and Prendiville \cite{BP2016} have shown Green's transference method to be versatile, establishing a theorem of Roth--Waring type for five squares. They were able to transfer results from the integers to the squares, thereby obtaining a Roth-type bound for a quadratic equation without the property of translation-dilation invariance. The present article combines aspects of Roth's theorem, Waring's problem and Goldbach problems.

Let $c_1, \ldots, c_s$ be nonzero integers such that
\begin{equation} \label{zerosum}
c_1 + \ldots + c_s = 0.
\end{equation}
Let $K$ be a union of $k$ proper subspaces of the rational hyperplane
\begin{equation} \label{LinearEquation}
c_1 x_1 + \ldots + c_s x_s = 0,
\end{equation}
each of which contains the diagonal 
\begin{equation} \label{diagonal}
\{(x,\ldots,x): x \in \bQ\}.
\end{equation}
Let $d \ge 2$ be an integer, and let $A$ be a set of primes in $[X] := \{1,2,\ldots,X\}$ such that the only solutions $\bx \in A^s$ to
\begin{equation} \label{ActualEquation}
c_1 x_1^d + \ldots + c_s x_s^d = 0
\end{equation}
have $(x_1^d, \ldots, x_s^d) \in K$. 

\begin{thm} \label{thm1} Assume $s \ge C(d)$, where $C(2) = 5$, $C(3) = 9$, $C(4) = 15$ and
\[
C(d)=  d^2 + 1 \qquad (d \ge 5).
\]
Then
\begin{equation} \label{DensityBound}
|A| \ll_{\bc, k, \eps} \frac X{\log X} (\log \log \log \log X)^{\frac{2-s}d+\eps}.
\end{equation}
\end{thm}

Loosely, this says that any subset of relative density $(\log \log \log \log)^{(2-s)/d+\eps}$ within the primes contains nontrivial solutions to \eqref{ActualEquation}. One could choose 
\[
K = \bigcup_{i \ne j} \{ \bx \in \bQ^s: x_i = x_j, \: \bc \cdot \bx = 0 \},
\]
for instance, which was the choice of Keil \cite{Kei2014} and Henriot \cite{Hen2015} in their work on diagonal quadrics in dense variables. Thus, any positive density subset of the primes contains a solution to \eqref{ActualEquation} with pairwise distinct coordinates. The reason for having a notion of trivial solutions is that, by \eqref{zerosum}, the diagonal \eqref{diagonal} lies within the solution set of \eqref{ActualEquation}.

We shall use Green's transference technology \cite{Gre2005} to transfer Roth-type results from the integers to the set of $d$th powers of primes. The protagonist shall be a measure $\nu$ on some interval $[N]$, where $N$ can be thought of as $X^d$. Morally $\nu(n)$ should be $dp^{d-1} \log p$, if $n = p^d$ for some prime $p \le X$, and zero otherwise. The measure $\nu$ has become known as a majorant; a \emph{majorant} on $[N]$ is a function $\nu: \bZ \to [0,\infty)$  with support in $[N]$. Our majorant shall have the additional normalisation property that 
\begin{equation} \label{normalisation}
\| \nu\|_1 \sim N.
\end{equation}
We refer the curious reader to the expository article \cite{Pre2015} for more on the history and terminology of the transference principle.

The point is that our set $A$ can be lifted to $[N]$ and weighted by $\nu$ to behave like a dense subset --- not of $\cP_X := \{ \text{prime } p \le X\}$, but of $[N]$. Bloom's theorem \cite{Blo2012} then ensures that the $\nu$-weighted solution count is large in terms of the density --- see \cite[\S 2]{BP2016} and \cite[\S 1.2]{Pre2015}. Since $A$ has only $K$-trivial solutions to \eqref{ActualEquation}, in the sense that the only solutions $\bx \in A^s$ to \eqref{ActualEquation} have $(x_1^d, \ldots,x_s^d) \in K$, we also obtain an upper bound for this count. Combining the two inequalities reaps a density bound of the shape \eqref{DensityBound}.

Browning and Prendiville \cite{BP2016} have distilled the method into the following ingredients.
\begin{enumerate}
\item Density transfer. We shall lift our set $A \subset \cP_X$ to a set $\cA \subset [N]$ in the support of our majorant $\nu$. With
\begin{equation} \label{delta}
\del := |A| \frac{\log X}X,
\end{equation}
we will show that $\cA$ has a $\nu$-weighted density of at least $\del^d$ in $[N]$. In other words, we shall establish the bound
\begin{equation} \label{DensityTransfer}
\sum_{n \in \cA} \nu(n) \gg \del^d N.
\end{equation}
\item Fourier decay. The majorant $\nu$ has \emph{Fourier decay of level $\tet$} if
\[
\| \hat{\nu} - \widehat{1_{[N]}} \|_\infty \le \tet N.
\]
We shall demonstrate a quantitatively $o(1)$ level of Fourier decay.
\item Restriction estimate. The majorant $\nu$ satisfies a \emph{restriction estimate at exponent $u$} if
\[
\sup_{|\phi| \le \nu} \int_\bT |\hat{\phi}(\alp)|^u \d \alp \ll_u \| \nu \|_1^u N^{-1}.
\]
\item $K$-trivial saving. For $\eta > 0$, the majorant $\nu$ \emph{saves $\eta$ on $K$-trivial solutions} if
\[
\sum_{\bx \in K} \prod_{i=1}^s \nu(x_i) \ll_{k,s,d,\eta} \| \nu \|_1^s N^{-1-\eta}.
\]
\end{enumerate}
Of these, the most technically demanding are Fourier decay and the restriction estimate, especially the latter. For both, it is necessary to have good pointwise estimates for certain exponential sums over primes. These exponential sums are given by the Fourier transform of our majorant. To ensure the necessary Fourier decay, we shall use the $W$-trick \cite{Gre2005}, which circumvents technical difficulties arising from the fact that the prime $d$th powers are not equidistributed in congruence classes to small moduli.

The number of variables required in Theorem \ref{thm1} is determined by the restriction estimate. Since we have good control on the growth of the weights involved, we shall see that the restriction estimate can be derived from a moment estimate for a simpler exponential sum. This leads to the following strengthening of Theorem \ref{thm1}. 

\begin{thm} \label{MainThm} Let $t \ge d$ be an integer such that the number of solutions $\bz \in [X]^{2t}$ to
\begin{equation} \label{hyp}
z_1^d + \ldots + z_t^d = z_{t+1}^d + \ldots + z_{2t}^d
\end{equation}
is $O_{t,d,\eps}(X^{2t-d+\eps})$, and assume $s > 2t$. Then we have \eqref{DensityBound}.
\end{thm}

To deduce Theorem \ref{thm1}, we apply Theorem \ref{MainThm} with the particular choice
\begin{equation} \label{choice}
t = \begin{cases}
\lfloor d^2/2 \rfloor, &\text{if } d \ne 4 \\
7, &\text{if } d=4.
\end{cases}
\end{equation}
For this to be valid, we need to check that if $t$ is given by \eqref{choice} then the number of solutions $\bz \in [X]^{2t}$ to \eqref{hyp} is indeed $O_{t,d,\eps}(X^{2t-d+\eps})$. For $(d,t) = (2,2)$, it is known that \eqref{hyp} has $O(X^2 \log X)$ solutions $\bz \in [X]^4$. For $(d,t) = (3,4)$ it is known that there are $O(X^5)$ solutions --- this follows, for instance, from the methods of \cite{Vau1986}.

For $d \ge 5$, one can show that if $t$ is given by \eqref{choice} then \eqref{hyp} has $O(X^{2t-d})$ solutions $\bz \in [X]^{2t}$. This is a consequence of the main conjecture in Vinogradov's mean value theorem, which was recently established by Bourgain, Demeter and Guth \cite{BDG2016}. Indeed, the equation \eqref{hyp} is a direct analogue of the equation
\begin{equation} \label{DefiniteVersion}
z_1^d + \ldots + z_{2t}^d = n \qquad (n \text{ large and fixed})
\end{equation}
addressed in \cite[Theorem 4.1]{Woo2012}. That theorem, which was previously conditional on \cite{BDG2016}, tells us that \eqref{DefiniteVersion} has $O(X^{2t-d})$ solutions $\bz \in [X]^{2t}$, since
\[
2t \ge d^2 - 1 \ge d^2 + 1 - \Biggl \lfloor \frac{\log d}{\log 2} \Biggr \rfloor.
\]
One can follow the proof of \cite[Theorem 4.1]{Woo2012} verbatim, to show that \eqref{hyp} has $O(X^{2t-d})$ solutions $\bz \in [X]^{2t}$. The point is that when $s \ge \frac12 d(d+1)$, the displayed equation in that proof tells us that we save $X^{d+1-\eps}$ on the ($2s$)th moment on minor arcs, which is far more than the $X^{d+\eps}$ that we need to save; the interpolation procedure in \cite[\S 3]{Woo2012} then secures a saving of $X^{d+\eps}$ on the ($2t$)th moment.

For $(d,t) = (4,7)$, we again follow the proof of \cite[Theorem 4.1]{Woo2012}, interpolating on minor arcs between an eighth and a twentieth moment. Hua's lemma \cite[Lemma 2.5]{Vau1997} yields the eight moment bound $O(X^{5+\eps})$. On minor arcs, the equation displayed in the proof of \cite[Theorem 4.1]{Woo2012} gives us the twentieth moment bound $O(X^{15+\eps})$. The treatment of the major arcs is standard \cite[\S 4.4]{Vau1997}, and we conclude that the number $\cN$ of solutions $\bz \in [X]^{14}$ to \eqref{hyp} satisfies
\[
\cN \ll (X^{5+\eps})^{1/2} (X^{15+\eps})^{1/2} = X^{10+\eps}.
\]

Theorem \ref{MainThm} also enables a septenary result for cubes, assuming the so-called Hooley Riemann hypothesis (HRH); see \cite[\S 6]{BW2014}. The statement below follows easily from Theorem \ref{MainThm} and \cite[Lemma 6.2]{BW2014}.

\begin{cor} \label{Hooley} Assume HRH, $d = 3$ and $s \ge 7$. Then we have \eqref{DensityBound}.
\end{cor}

We now comment on the relevance of restriction theory \cite{Bou1989, Bou1993, Hen2016, Tao2004}. The number of variables required to implement the circle method is often governed by the exponent at which we know a sharp moment estimate for an exponential sum. When the variables are restricted to lie in a set $A$, the relevant exponential sums necessarily come with weights supported on $A$. The key ingredient for such problems, therefore, is a moment estimate for an exponential sum with fairly arbitrary weights. Restriction theory concerns inequalities between norms of Fourier transforms, which is the same as bounding moments of weighted exponential sums.

Finally, we feel it is appropriate to describe the difficulties involved in proving the restriction estimate, and to outline our strategy for doing so. Using our hypothesis on $t$, step one is to deduce an `almost-sharp' restriction estimate at exponent $2t$: this fails to be sharp by a factor of $X^\eps$ (see Lemma \ref{intermediate}). If one could obtain a power saving on traditional minor arcs, for Weyl sums over primes, then a standard epsilon-removal process would complete the proof (see \cite[\S 4]{Bou1989}); the problem is that current technology only allows us to save a logarithmic factor here (see Lemma \ref{minorbound}). Our approach is to introduce a related majorant $\mu$, not involving primes. Step two is to establish a restriction inequality for $\mu$ at some intermediate exponent, sharp up to a logarithmic factor (see Lemma \ref{platform}). Step three is to bootstrap this to the full-strength restriction estimate for our prime power majorant $\nu$.

Note that for the specific value \eqref{choice} of $t$, we know that the equation \eqref{hyp} has $O(X^{2t-d} \log X)$ solutions, at least when $d \ne 4$: recall the discussion following \eqref{choice}. Thus, when $d \ne 4$, the restriction estimate relevant to Theorem \ref{thm1} is substantially easier to prove --- one deals with the logarithmic factor directly using the methods in \cite[\S 4]{Bou1989}. Our more general approach reduces the number of variables required when $d=4$, enables Corollary \ref{Hooley}, and also anticipates future improvements in our understanding of the diophantine equation \eqref{hyp}.

We organise thus. In \S \ref{W}, we shall construct our majorant $\nu$, confirm \eqref{normalisation}, define our lifted set $\cA$, and establish the density transfer inequality \eqref{DensityTransfer}. In \S \ref{exponential}, we use the circle method to study the Fourier transform $\hat{\nu}$. The analysis therein will allow us to establish Fourier decay in \S \ref{decay}. In \S \ref{Restriction}, we use Bourgain's methods \cite{Bou1989} to prove that $\nu$ satisfies the relevant restriction estimate. We check in \S \ref{trivial} that $\nu$ saves $1/s$ on $K$-trivial solutions, before putting it all together to prove Theorem \ref{MainThm} in \S \ref{density}.

We adopt the convention that $\eps$ denotes an arbitrarily small positive real number, so its value may differ between instances. The symbol $p$ shall be reserved for primes. For $x \in \bR$ and $q \in \bN$, put \mbox{$e(x) = e^{2 \pi i x}$} and $e_q(x) = e^{2 \pi i x / q}$. Boldface will be used for vectors, for instance we abbreviate $(x_1,\ldots,x_n)$ to $\bx$, and define $|\bx| = \max(|x_1|, \ldots, |x_n|)$. For $x \in \bR$, let $\| x \|$ be the distance from $x$ to the nearest integer. Let $\cP$ denote the set of primes. For $Y \in \bN$, let $[Y] = \{ 1,2, \ldots, Y \}$ and $\cP_Y = \cP \cap [Y]$. We shall make use of the offset logarithmic integral $\Li(x) = \int_2^x \frac {\d t}{\log t}$.

We write $\bT$ for the torus $\bR / \bZ$. We shall use Landau and Vinogradov notation: for functions $f$ and positive-valued functions $g$, write $f \ll g$ or $f = O(g)$ if there exists a constant $C$ such that $|f(x)| \le C g(x)$ for all $x$. If $S$ is a set, we denote the cardinality of $S$ by $|S|$ or $\# S$. The pronumeral $X$ denotes a large positive integer, and we shall put $L = \log X$ throughout. We write $C_1, C_2, \ldots$ for positive constants that appear in the course of our proofs.

For $r \ge 1$ and $f: \bZ \to \bC$, we define the $L^r$-norm by
\[
\| f \|_r = \Bigl( \sum_n |f(n)|^r \Bigr)^{1/r}.
\]
When $\| f\|_1 < \infty$, we also define the Fourier transform of $f$ by
\begin{align*} \hat{f}: \bT &\to \bC, \\
\hat{f}(\alp) &= \sum_n f(n) e(\alp n).
\end{align*}

The author would like to thank his advisor Trevor Wooley very much for his guidance. Thanks also to Tim Browning and Sean Prendiville for fruitful conversations. We thank the anonymous referee for a detailed review containing several helpful suggestions.

\section{The $W$-trick}
\label{W}

We begin by defining our majorant $\nu$. As discussed, we shall apply the $W$-trick \cite{Gre2005} from the outset, so that we will later obtain sufficient Fourier decay. Let
\begin{equation} \label{Wdef}
w = \frac12 \log \log X, \qquad W= 4d^3 \prod_{p \le w} p.
\end{equation}
Note the factor of $d^3$ included in the definition \eqref{Wdef} of $W$; this special feature will come into play during the case analysis in \S \ref{decay}. Since $X$ is large, it follows from the prime number theorem that
\begin{equation} \label{Wbound}
W \le e^{2w} = \log X = L.
\end{equation}
For $b \in [W]$ with
\begin{equation} \label{bcond}
-b \in (\bZ/W\bZ)^{\times d} := \{ z^d: z \in (\bZ / W \bZ)^\times \},
\end{equation}
let
\begin{equation} \label{sigdef}
\sig(b) = \# \{z \in [W]: z^d \equiv -b \mmod \it{W}\}.
\end{equation}

We begin with the observation that $\sig(b)$ does not, in fact, depend on $b$. To verify this it suffices, by the Chinese remainder theorem, to show that if $r \in \bN$ and $p^r \| W$ then 
\[
\# \{z \mmod p^r: z^d \equiv -b \mmod p^r \}
\]
is the same for each $b$ satisfying \eqref{bcond}. If $p \ne 2$ then this follows easily using a primitive root (see \cite[Ch. 10]{Apo1976}). For $p =2$, we can instead use the fact that any odd residue class is representable uniquely  as $(-1)^u 5^v$, with $u \mmod 2$ and $v \mmod 2^{r-2}$ (see \cite[Ch. 4]{Dav2000}). We conclude that $\sig(b)$ is the same for each $b \in -(\bZ/W\bZ)^{\times d}$.

To ensure density transfer, we shall choose $b$ to maximise the $\nu_b$-measure of $\cA_b$, where
\begin{equation} \label{nudef}
\nu_b(n) =
\begin{cases}
\frac {\varphi(W)}{W \sig(b)} dp^{d-1} \log p, & \text{if } Wn-b = p^d \text{ with } p \in \cP_X \\
0, &\text{ otherwise}
\end{cases}
\end{equation}
and
\begin{equation} \label{cAdef}
\cA_b = \{ n \in \bZ: Wn-b = p^d \text{ for some } p \in A \}.
\end{equation}
Let 
\begin{equation} \label{Ndef}
N = \lfloor X^d/W \rfloor + 1.
\end{equation}

\begin{lemma} [Density transfer] \label{DT} Assume $\del > (\log X)^{-1}$. Then there exists $b \in [W]$ such that $-b \in (\bZ/W\bZ)^{\times d}$ and
\[
\sum_{n \in \cA_b} \nu_b(n) \gg \del^d N.
\]
\end{lemma}

\begin{proof} We shall implicitly embed $-(\bZ/W\bZ)^{\times d}$ into $[W]$, in the obvious way. We use a standard averaging argument, noting first that
\[
\sum_{b \in -(\bZ/W\bZ)^{\times d} } \sum_{n \in \cA_b} \nu_b(n)
= \sum_{p \in A} \frac {\varphi(W)}{W \sig(b)} dp^{d-1} \log p 
- \sum_{p \in A, \: p \le w} \frac {\varphi(W)}{W \sig(b)} dp^{d-1} \log p.
\]
By over-counting, and by recalling that $\sig(b)$ is the same for each $b \in -(\bZ/W\bZ)^{\times d}$, we deduce that
\[
|-(\bZ/W\bZ)^{\times d}|= \frac{\varphi(W)}{\sig(b)}.
\]
Thus, if $b$ is chosen to maximise $\sum\limits_{n \in \cA_b} \nu_b(n)$, then
\[
\sum_{n \in \cA_b} \nu_b(n)
\ge -1 + \sum_{p \in A} W^{-1} dp^{d-1} \log p.
\]

A crude lower bound for $\sum\limits_{p \in A} p^{d-1} \log p$ is given by the sum of $p^{d-1} \log p$ over the first $|A|$ primes $p$. By \eqref{delta} and the prime number theorem, we now have
\[
\sum_{p \in A} p^{d-1} \log p \ge \sum_{p \le (1-\eps) \del X} p^{d-1} \log p \gg \del^d X^d.
\]
Hence
\[
\sum_{n \in \cA_b} \nu_b(n) \gg \del^d X^d / W \gg \del^d N.
\]
\end{proof}

The assumption that $\del > (\log X)^{-1}$ is harmless in the context of Theorem \ref{MainThm}, for if $\del \le (\log X)^{-1}$ then we certainly have \eqref{DensityBound}. We henceforth fix $b$ as in Lemma \ref{DT}, and write
\begin{equation} \label{fixb}
\cA = \cA_b, \qquad \nu = \nu_b,
\end{equation}
so that we have \eqref{DensityTransfer}. Note that our majorant $\nu$ is supported on $[N]$, and that
\begin{equation} \label{supremum}
\| \nu \|_\infty \ll X^{d-1} L.
\end{equation}
Next, we verify \eqref{normalisation}. The proof is standard, but we nonetheless present it, as it will prepare us well for the next section.

We compute:
\begin{align} \notag
\| \nu \|_1 &= \sum_{\substack{p \le X: \\ p^d \equiv -b \mmod \it{W}}} \frac {\varphi(W)}{W \sig(b)} dp^{d-1} \log p \\
\label{L1calc} &= \sig(b)^{-1} \sum_{\substack{z \in [W]: \\ z^d \equiv -b \mmod \it{W}}} \frac{\varphi(W)}W 
\sum_{\substack{p \le X: \\ p \equiv z \mmod \it{W}}} dp^{d-1} \log p.
\end{align}
The inner sum is treated using Abel summation. For $n \in [X]$, put
\[
A_n = \sum_{\substack{p \le n: \\ p \equiv z \mmod \it{W}}} 1.
\]
Note that since $(b,W) = 1$ we must also have $(z,W) = 1$. The bound \eqref{Wbound} allows us to apply Siegel--Walfisz \cite[Lemma 7.14]{Hua1965}, so
\[
A_n = \frac{\Li(n)}{\varphi(W)} + O(X e^{-C_1 \sqrt{L}}).
\]

With $g(n) = dn^{d-1} \log n$, we have
\begin{align*}
\sum_{\substack{p \le X: \\ p \equiv z \mmod \it{W}}} dp^{d-1} \log p 
&= \sum_{n =2}^X (A_n - A_{n-1}) g(n) 
\\ &= A_X g(X+1) + \sum_{n=2}^X A_n (g(n) - g(n+1)).
\end{align*}
The mean value theorem tells us that $g(n) - g(n+1) \ll X^{d-2}L$. In light of \eqref{Wbound} and \eqref{Ndef}, we now have
\begin{align*}
\varphi(W) \sum_{\substack{p \le X: \\ p \equiv z \mmod \it{W}}} dp^{d-1} \log p 
&= \Li(X)g(X+1) + \sum_{n=2}^X \Li(n) \cdot (g(n) - g(n+1)) \\
&\qquad + O(Ne^{-C_2 \sqrt{L}}).
\end{align*}
As $\Li(2) = 0$, we therefore have
\begin{align*}
\varphi(W) \sum_{\substack{p \le X: \\ p \equiv z \mmod \it{W}}} dp^{d-1} \log p 
&= \sum_{n=3}^X g(n) \int_{n-1}^n \frac{\d x}{\log x}+ O(Ne^{-C_2 \sqrt{L}}) \\
&= \sum_{n=3}^X dn^{d-1} \int_{n-1}^n \frac{ \log n}{\log x} \d x+ O(Ne^{-C_2 \sqrt{L}}).
\end{align*}

When $2 \le n-1 < x < n$, the mean value theorem tells us that 
\[
\log n = \log x + O(1/n).
\]
Hence
\[
\varphi(W) \sum_{\substack{p \le X: \\ p \equiv z \mmod \it{W}}} dp^{d-1} \log p 
= \sum_{n \le X} dn^{d-1} + O(Ne^{-C_2 \sqrt{L}}) = X^d + O(Ne^{-C_2 \sqrt{L}}).
\]
Substituting this into \eqref{L1calc}, and recalling \eqref{sigdef} and \eqref{Ndef}, yields
\[
\| \nu \|_1 = X^d/W + O(Ne^{-C_2 \sqrt{L}}) = N + O(Ne^{-C_2 \sqrt{L}}),
\]
confirming \eqref{normalisation}.

\section{Exponential sums}
\label{exponential}

We wish to investigate
\begin{align}
\notag \hat{\nu}(\alp) &= \frac{\varphi(W)}{W\sig(b)} \sum_{\substack{p \le X:\\ p^d \equiv -b \mmod \it{W}}} (dp^{d-1} \log p) e(\alp (p^d+b)/W) \\
\label{hatcalc} &= \frac{\varphi(W) e(\alp b/W)}{W\sig(b)}\sum_{\substack{z \in [W]: \\ z^d \equiv -b \mmod \it{W}}}
\:\: \sum_{\substack{p \le X: \\ p \equiv z \mmod \it{W}}} (dp^{d-1} \log p) e(\alp p^d/W),
\end{align}
so we focus on the inner sum. We begin with a Hardy--Littlewood dissection, dissecting $\bT$ into major arcs $\fM$ and minor arcs $\fm$. Let $\sig_0$ be a large positive constant, and let $\sig$ be a much larger positive constant. For $q \in \bN$ and $a \in \bZ$, let $\fM(q,a)$ be the set of $\alp \in \bT$ such that $|\alp - a/q| \le L^\sig X^{-d}$. Let $\fM(q)$ be the union of the sets $\fM(q,a)$ over integers $a$ such that $(a,q) = 1$, and let $\fM$ be the union of the sets $\fM(q)$ over $q \le L^\sig$. Put $\fm = \bT \setminus \fM$. By identifying $\bT$ with a unit interval, we may write\
\[
\fM(q) = \bigcup_{\substack{{a=0}\\(a,q)=1}}^{q-1} \fM(q,a).
\]

\begin{lemma} \label{minorbound} If $\alp \in \fm$ then $\hat{\nu}(\alp) \ll NL^{-\sig_0}$.
\end{lemma}

\begin{proof} Let $\alp \in \fm$. By Dirichlet's approximation theorem \cite[Lemma 2.1]{Vau1997}, we obtain relatively prime integers $q$ and $a$ such that $1 \le q \le X^d L^{-\sig}$ and $|q \alp - a| \le L^\sig X^{-d}$. Now $|\alp-a/q| \le L^\sig X^{-d}$ so, as $\alp \notin \fM$, we must have $q > L^\sig$. Thus, with $\bet = \alp - a/q$, we have
\[
|\bet| \le \frac{L^\sig}{qX^d} \le X^{-d}.
\]
Let $z \in [W]$ with $z^d \equiv -b \mmod \it{W}$. By partial summation, we have
\begin{align*}
\sum_{\substack{p \le X: \\ p \equiv z \mmod \it{W}}} (dp^{d-1} \log p) e(\alp p^d/W)
&= \sum_{\substack{p \le X: \\ p \equiv z \mmod \it{W}}} e(\bet p^d/W) (dp^{d-1} \log p) e_{Wq}(a p^d) \\
&= A_\diam(X) f(X) - \int_1^X A_\diam(t) f'(t) \d t,
\end{align*}
where $f(t) = e(\bet t^d/W) dt^{d-1} \log t$ and
\[
A_\diam(t) = \sum_{\substack{p \le t: \\ p \equiv z \mmod \it{W}}} e_{Wq}(ap^d).
\]

Note that $f'(t) \ll X^{d-2}L$, and that $|A_\diam(t)| \le t$. Thus,
\begin{align*}
\sum_{\substack{p \le X: \\ p \equiv z \mmod \it{W}}} (dp^{d-1} \log p) e(\alp p^d/W)
&= A_\diam(X) f(X) - \int_{XL^{-2\sig_0}}^X A_\diam(t) f'(t) \d t 
\\ &\quad+ O(NL^{-\sig_0}).
\end{align*}
In view of \eqref{sigdef} and \eqref{hatcalc}, we now have
\[
\hat{\nu}(\alp) \ll NL^{-\sig_0} + X^{d-1}L \sup_{XL^{-2\sig_0} < t \le X} |A_\diam(t)|.
\]
It remains to estimate $A_\diam(t)$ when $XL^{-2\sig_0} < t \le X$. We shall use \cite[Theorem 10]{Hua1965} for this. In order to apply this result, we need to control size of the denominator
\[
q^* := \frac{Wq}{(a,Wq)} = \frac{Wq}{(a,W)},
\]
in terms of $t$. Recalling \eqref{Wbound}, we have
\[
L^\sig < q \le q^* \le Wq \le X^d L^{1-\sig} \le t^d L^{-\sig/2}.
\]
As $L \ll \log t \ll L$, we may thus invoke \cite[Theorem 10]{Hua1965}, which tells us that
\[
A_\diam(t) \ll XL^{-\sig_0-1}W^{-1}.
\]
Hence $\hat{\nu}(\alp) \ll NL^{-\sig_0}$.
\end{proof}

On major arcs we can decompose our Fourier transform into archimedean and non-archimedean components. When $(z,W) = 1$, let 
\begin{equation} \label{Sdef}
 S^*_q(a,z) = \sum_{\substack{r \mmod \it{q}: \\ (z+Wr,Wq)=1}} e_q\Bigl(a \frac{(z+Wr)^d+b}W \Bigr)
\end{equation}
and
\begin{equation} \label{Idef}
I(\bet) = \int_0^N e(\bet t) \d t.
\end{equation}

\begin{lemma} 
Let $\alp \in \fM(q,a)$ with $(a,q) = 1$ and $q \le L^\sig$, and put
\begin{equation} \label{beta}
\bet = \alp - a/q \in [-L^\sig X^{-d}, L^\sig X^{-d}].
\end{equation}
Then
\begin{equation} \label{majorarcs} 
\hat{\nu}(\alp) = I(\bet) \sig(b)^{-1} \sum_{\substack{z \in [W]: \\ z^d \equiv -b \mmod \it{W}}}  \frac{\varphi(W)}{\varphi(Wq)}  S^*_q(a,z) + O(Ne^{-C_4 \sqrt{L}}).
\end{equation}
\end{lemma}

\begin{proof} With \eqref{hatcalc} in mind, we initially fix $z \in [W]$ with $z^d \equiv -b \mmod \it{W}$, and study
\[
\sum_{\substack{p \le X: \\ p \equiv z \mmod \it{W}}} (dp^{d-1} \log p) e(\alp p^d/W).
\]
For $n \in [X]$, let
\[
S_n = \sum_{\substack{p \le n: \\ p \equiv z \mmod \it{W}}} e_{Wq}(ap^d).
\]
Then
\[
S_n = O(Wq) + \sum_{\substack{r \mmod q: \\ (z+Wr, Wq) = 1}} e_{Wq}(a(z+Wr)^d) \sum_{\substack{p \le n: \\ p \equiv z+Wr \mmod \it{Wq}}} 1.
\]
As $n \le X$ and $Wq \le L^{\sig+1}$, the inner sum is amenable to Siegel--Walfisz \cite[Lemma 7.14]{Hua1965}, and so
\[
S_n = \frac{\Li(n)}{\varphi(Wq)} V_q(a,z) + O(Xe^{-C_3 \sqrt{L}}),
\]
where
\begin{equation} \label{Vdef}
V_q(a,z) = \sum_{\substack{r \mmod q: \\ (z+Wr, Wq) = 1}} e_{Wq}(a(z+Wr)^d) .
\end{equation}

With $f(t) = e(\bet t^d/W) dt^{d-1} \log t$, we have
\begin{align*}
\sum_{\substack{p \le X: \\ p \equiv z \mmod \it{W}}} (dp^{d-1} \log p) e(\alp p^d/W)
&= \sum_{n=2}^X (S_n - S_{n-1}) f(n) \\
&=  S_X f(X+1) + \sum_{n=2}^X S_n (f(n) - f(n+1)).
\end{align*}
As $|\bet| \le L^\sig X^{-d}$, the mean value theorem implies that 
\[
f(n) - f(n+1) \ll X^{d-2} L^{\sig+1}.
\]
Hence
\begin{align*}
& \sum_{\substack{p \le X: \\ p \equiv z \mmod \it{W}}} (dp^{d-1} \log p) e(\alp p^d/W) \\
& = \frac{V_q(a,z)}{\varphi(Wq)} \Bigl[
\Li(X)f(X+1) + \sum_{n=2}^X \Li(n) \cdot (f(n) - f(n+1))
\Bigr]
+O(Ne^{-C_4 \sqrt{L}}).
\end{align*}
As $\Li(2) = 0$, we thus have
\begin{align} \notag
&\sum_{\substack{p \le X: \\ p \equiv z \mmod \it{W}}} (dp^{d-1} \log p) e(\alp p^d/W)\\
\label{innercalc} &= \frac{V_q(a,z)}{\varphi(Wq)} \sum_{n=3}^X \int_{n-1}^n \frac{f(n)} {\log x} \d x
+O(Ne^{-C_4 \sqrt{L}}).
\end{align}

When $n-1 < x < n$, the mean value theorem reveals that 
\[
f(n) = f(x) + O(X^{d-2}L^{\sig+1}),
\]
and so
\begin{align*}
\sum_{n=3}^X \int_{n-1}^n \frac{f(n)} {\log x} \d x &= \int_2^X dx^{d-1} e(\bet x^d/W) \d x + O(X^{d-1}L^{\sig+1}) \\
&= WI(\bet) +  O(X^{d-1} L^{\sig+1}).
\end{align*}
Substituting this into \eqref{innercalc}, and noting that $|V_q(a,z)| \le q \le L^\sig$, gives
\[
\sum_{\substack{p \le X: \\ p \equiv z \mmod \it{W}}} (dp^{d-1} \log p) e(\alp p^d/W) \\
= \frac{W}{\varphi(Wq)} V_q(a,z) I(\bet)+O(Ne^{-C_4 \sqrt{L}}).
\]
Substituting this into \eqref{hatcalc}, and recalling \eqref{sigdef}, gives
\[
\hat{\nu}(\alp) = \frac{\varphi(W) e(\alp b/W)}{\varphi(Wq) \sig(b)} \sum_{\substack{z \in [W]: \\ z^d \equiv -b \mmod \it{W}}}  V_q(a,z) I(\bet) +O(Ne^{-C_4 \sqrt{L}}).
\]

From \eqref{Sdef} and \eqref{Vdef}, we see that
\[
e_{Wq}(ab) V_q(a,z) = S_q^*(a,z).
\]
Since
\[
e(\bet b/W) I(\bet) = \int_0^N e(\bet (t+b/W)) \d t = I(\bet) + O(1),
\]
we finally have \eqref{majorarcs}.
\end{proof}

\section{Fourier decay}
\label{decay}

In this section, we will establish that $\nu$ has Fourier decay of level $O(w^{\eps - 1/2})$. In other words, we shall prove that if $\alp \in \bT$ then 
\begin{equation} \label{FourierDecay}
\hat{\nu}(\alp) - \widehat{1_{[N]}} (\alp) \ll_\eps w^{\eps-1/2}N.
\end{equation}
By a geometric series, we have
\begin{equation} \label{geom}
\widehat{1_{[N]}} (\alp) = \sum_{n \le N} e(\alp n) \ll \| \alp \|^{-1}.
\end{equation}

First suppose $\alp \in \fm$. By Dirichlet's approximation theorem, we obtain relatively prime integers $q$ and $a$ such that $1 \le q \le L^\sig$ and \mbox{$|q \alp - a| \le L^{-\sig}$.} As $\alp \notin \fM$, we must have $|\alp - a/q| > L^\sig X^{-d}$, so 
\[
\widehat{1_{[N]}} (\alp) \ll \| \alp \|^{-1} \ll \frac q { \| q \alp \|} \ll X^d L^{-\sig}.
\]
Recalling \eqref{Wbound} and \eqref{Ndef}, we now have $\widehat{1_{[N]}} (\alp) \ll N L^{1-\sig}$. Coupling this with Lemma \ref{minorbound}, using the triangle inequality, yields
\[
\hat\nu (\alp) - \widehat{1_{[N]}} (\alp) \ll N L^{1-\sig} + NL^{-\sig_0}.
\]
Upon recalling the definition \eqref{Wdef} of $w$, we conclude that \eqref{FourierDecay} holds for $\alp \in \fm$.

Next we consider the case in which $q=1$ and $\alp \in \fM(q)$. In other words,
\[
\|\alp\| \le L^\sig X^{-d}.
\]
From \eqref{sigdef} and \eqref{majorarcs}, we see that
\[
\hat \nu (\alp) = I(\alp) + O(Ne^{-C_4 \sqrt L}).
\]
By Euler--Maclaurin summation \cite[Eq. (4.8)]{Vau1997}, we have
\[
\widehat{1_{[N]}} (\alp) - I(\alp) \ll 1 + N \|\alp\| \ll 1 + NL^\sig X^{-d} \ll L^\sig.
\]
The triangle inequality now gives
\[
\hat \nu (\alp) - \widehat{1_{[N]}} (\alp) \ll Ne^{-C_4 \sqrt L}.
\]
Recalling \eqref{Wdef}, we conclude that \eqref{FourierDecay} holds whenever $\alp \in \fM(1)$.

Finally, let $\alp \in \fM(q,a)$ with $2 \le q \le L^\sig$ and $(a,q) = 1$, and put \eqref{beta}. Since $q \ge 2$, we must have $|a| \ge 1$. Substituting
\[
\| \alp \| \ge q^{-1} - |\bet| \ge q^{-1} - L^\sig X^{-d} \gg q^{-1}
\]
into \eqref{geom} gives
\begin{equation} \label{basic}
\widehat{1_{[N]}} (\alp) \ll q \ll L^\sig.
\end{equation}
By \eqref{sigdef}, \eqref{majorarcs} and the trivial estimate $|I(\bet)| \le N$, we have
\begin{equation} \label{majorcalc}
\hat \nu (\alp) \ll Ne^{-C_4 \sqrt L} + \frac{N}{\varphi(q)} \sup_z |S_q^*(a,z)|,
\end{equation}
where the supremum is over $z \in [W]$ such that $(z,W) = 1$.

We now study the sums $S_q^*(a,z)$. Let $z \in [W]$ with $(z,W) = 1$. By \eqref{Sdef}, we have
\begin{equation} \label{twist}
S_q^*(a,z) = e_{Wq}(a(z^d+b)) S^\diam_q(a,z),
\end{equation}
where 
\[
S^\diam_q(a,z) = \sum_{\substack{r \mmod \it{q}: \\ (z+Wr,Wq)=1}} e_q \Bigl(a
\sum_{\ell = 1}^d {d \choose \ell} W^{\ell - 1} z^{d- \ell} r^\ell
\Bigr).
\]
As $(z+Wr,W) = (z,W) = 1$, we have the slightly simpler expression
\[
S^\diam_q(a,z) = \sum_{\substack{r \mmod \it{q}: \\ (z+Wr,q)=1}} e_q \Bigl(a
\sum_{\ell = 1}^d {d \choose \ell} W^{\ell - 1} z^{d- \ell} r^\ell
\Bigr).
\]

Let $q = uv$, where $u$ is $w$-smooth and $(v,W) = 1$. Since $(u,v) = 1$, a standard calculation reveals that
\begin{equation} \label{decomp}
S^\diam_q(a,z) = S^\diam_u(a_1,z) S^\diam_v(a_2,z),
\end{equation}
where $a_1 = av^{-1} \in (\bZ / u\bZ)^\times$ and $a_2 = au^{-1} \in (\bZ / v\bZ)^\times$ (see \cite[Lemma 2.10]{Vau1997}). First consider
\[
S^\diam_u(a_1,z) = \sum_{\substack{r \mmod \it{u}: \\ (z+Wr,u)=1}} e_u \Bigl(a_1
\sum_{\ell = 1}^d {d \choose \ell} W^{\ell - 1} z^{d- \ell} r^\ell
\Bigr).
\]
As $u$ is $w$-smooth and $(z,W) = 1$, the condition $(z+Wr,u)=1$ is always met, and so
\begin{equation} \label{udiam}
S^\diam_u(a_1,z) = \sum_{r \mmod \it{u}} e_u \Bigl(a_1
\sum_{\ell = 1}^d {d \choose \ell} W^{\ell - 1} z^{d- \ell} r^\ell
\Bigr).
\end{equation}

We now borrow a strategy employed in \cite[\S 5]{BP2016}. Let $h = (u,W)$, and put $u = hu'$ and $W= hW'$, noting that $(u',W') = 1$. Writing $r = r_1 + u' r_2$, with $r_1 \mmod u'$ and $r_2 \mmod h$, yields
\begin{align*}
S^\diam_u(a_1,z) &=
\sum_{\substack{r_1 \mmod u' \\ r_2 \mmod h}}
e_{hu'} \Bigl(a_1
\sum_{\ell = 1}^d {d \choose \ell} (hW')^{\ell - 1} z^{d- \ell} (r_1+u'r_2)^\ell
\Bigr) \\
&= \sum_{r_1 =0}^{u'-1} e_{hu'} \Bigl(a_1
\sum_{\ell = 1}^d {d \choose \ell} (hW')^{\ell - 1} z^{d- \ell} r_1^\ell
\Bigr) \\
& \qquad \sum_{r_2 = 0}^{h-1} e_h  \Bigl(a_1
\sum_{\ell = 1}^d {d \choose \ell} (hW')^{\ell - 1} z^{d- \ell} (u')^{\ell-1} r_2^\ell
\Bigr).
\end{align*}
The inner sum is
\[
\sum_{r_2 \mmod h} e_h(da_1 z^{d-1} r_2),
\]
which vanishes unless $h \mid d a_1 z^{d-1}$. As $(h,a_1) = (h,z) = 1$, we conclude that
\begin{equation} \label{bigu}
S^\diam_u(a_1,z) =0, \quad \text{if } (u,W) \nmid d,
\end{equation}
while if $h \mid d$ then
\begin{equation} \label{critical}
S^\diam_u(a_1,z) = h \sum_{r_1 = 0}^{u'-1} e_u \Bigl(a_1
\sum_{\ell = 1}^d {d \choose \ell} W^{\ell - 1} z^{d- \ell} r_1^\ell
\Bigr).
\end{equation}

Next consider
\[
e_{Wv}(a_2z^d) S^\diam_v(a_2,z) = \sum_{\substack{r \mmod \it{v}: \\ (z+Wr,v)=1}} e_v\Bigl(a_2
\frac{(z+Wr)^d}W\Bigr).
\]
As $(v,W) = 1$, we can change variables by $t = zW^{-1} + r \in \bZ / v \bZ$, which gives
\[
e_{Wv}(a_2z^d) S^\diam_v(a_2,z) = \sum_{\substack{t \mmod \it{v}: \\ (t,v)=1}} e_v(a_2 W^{d-1} t^d).
\]
Since $(a_2 W^{d-1}, v) = 1$, we may apply \cite[Lemma 8.5]{Hua1965}, which tells us that 
\begin{equation} \label{vbound}
S^\diam_v(a_2,z)  \ll v^{1/2+\eps} \ll q^{1/2+\eps}.
\end{equation}

As $q \ge 2$, we must have (i) $u \nmid d$, (ii) $1 \ne q \mid d$, or (iii) $u \mid d$ and $q > w$. 
\textbf{Case: $u \nmid d$.} Suppose for a contradiction that $(u,W) \mid d$. Then for all primes $p$ we have
\[
\min(\ord_p(u), \ord_p(W)) \le \ord_p(d).
\]
Since $\ord_p(W) > \ord_p(d)$ whenever $p \le w$, and since $u$ is $w$-smooth, this tells us that $u \mid d$, contradicting this case. Hence $(u,W) \nmid d$, so by \eqref{bigu} we have $S^\diam_u(a_1,z) = 0$. Therefore $S^*_q(a,z)$ vanishes, by \eqref{twist} and \eqref{decomp}. Now \eqref{majorarcs} gives
\begin{equation} \label{simple}
\hat \nu (\alp) \ll Ne^{-C_4 \sqrt L}.
\end{equation}

\textbf{Case: $1 \ne q \mid d$.} In this case $v=1$ and $q = u$. Further,
\[
h = (u, W) = (q, W) = q,
\]
since $q \mid d \mid W$. So $u' = 1$, and from \eqref{critical} we see that $S^\diam_q(a,z) = q$. By \eqref{twist}, we therefore have
\[
\sum_{\substack{z \in [W]: \\ z^d \equiv -b \mmod \it{W}}} S^*_q(a,z) 
= q \sum_{\substack{z \in [W]: \\ z^d \equiv -b \mmod \it{W}}} e_{Wq}(a(z^d+b)).
\]
We will find that this sum vanishes, which is the point of the $W$-trick. Write
\[
z \equiv x + \frac W d y \mmod W
\]
with $x \mmod W/d$ and $y \mmod d$. Note that
\[
z^d \equiv \Bigl(x+\frac W d y \Bigr)^d \equiv x^d \mmod W,
\]
in view of the definition \eqref{Wdef} of $W$. Hence
\[
q^{-1}\sum_{\substack{z \in [W]: \\ z^d \equiv -b \mmod \it{W}}} S^*_q(a,z) 
= \sum_{\substack{x \in [W/d]:\\ x^d \equiv -b \mmod W}} e_{Wq}(a(x^d+b)) \sum_{y \mmod d} e_q(ax^{d-1}y).
\]
As $(b,W) = 1$ and $q \mid W$, we have $(x,q) = 1$. Since $q \ne 1$ and $(a,q) = 1$, and since $q \mid d$, we deduce that the inner sum vanishes. Therefore
\[
\sum_{\substack{z \in [W]: \\ z^d \equiv -b \mmod \it{W}}} S^*_q(a,z) = 0.
\]
Substituting this into \eqref{majorarcs} yields \eqref{simple}.

\textbf{Case: $u \mid d$ and $q > w$.} By \eqref{udiam}, we have
\[
|S_u^\diam(a_1, z)| \le u \le d \ll 1.
\]
Now \eqref{twist}, \eqref{decomp} and \eqref{vbound} give
\begin{equation} \label{sqrt}
S^*_q(a,z) \ll q^{1/2+\eps}.
\end{equation}
Substituting this into \eqref{majorcalc} yields
\[
\hat \nu (\alp) \ll Ne^{-C_4 \sqrt L} + q^{2\eps - 1/2}N \ll w^{2 \eps - 1/2}N.
\]

Recalling \eqref{Wdef}, we see that we have 
\[
\hat \nu (\alp) \ll w^{\eps - 1/2}N
\]
in all three cases. Coupling this with \eqref{basic} yields \eqref{FourierDecay}. We conclude that the majorant $\nu$ has Fourier decay of level $O(w^{\eps - 1/2})$.

Note that the inequality \eqref{sqrt} is valid in all three cases. We record the following estimate for later use.
\begin{lemma} Let $\alp \in \fM(q,a)$ with $1 \le q \le L^\sig$ and $(a,q) = 1$.
Then
\begin{equation} \label{record}
\hat \nu(\alp) \ll q^{\eps- 1/2} \min\{ N, | \alp - a/q | ^{-1} \} + Ne^{-C_4 \sqrt L}.
\end{equation}
\end{lemma}

\begin{proof} Put \eqref{beta}. By \eqref{sigdef} and \eqref{majorarcs}, we have
\[
\hat \nu (\alp) \ll \varphi(q)^{-1} \sup_z |I(\bet) S^*_q(a,z)| + Ne^{-C_4 \sqrt L}.
\]
The integral \eqref{Idef} admits the standard estimate
\[
I(\bet) \ll \min \{ N, \|\bet\|^{-1} \} = \min \{ N, | \alp-a/q |^{-1} \},
\]
so by \eqref{sqrt} we have
\begin{align*}
\hat \nu (\alp) &\ll \varphi(q)^{-1} q^{1/2+\eps} \min \{ N, | \alp-a/q |^{-1} \} + Ne^{-C_4 \sqrt L}  \\
&\ll q^{2\eps -1/2} \min \{ N, | \alp-a/q |^{-1} \} + Ne^{-C_4 \sqrt L}.
\end{align*}
\end{proof}

\section{The restriction estimate}
\label{Restriction}

Let $t \ge d$ be an integer such that the number of solutions $\bz \in [X]^{2t}$ to \eqref{hyp} is $O_{t,d,\eps}(X^{2t-d+\eps})$. In this section, we show that $\nu$ satisfies a restriction estimate at any exponent $u > 2t$. The following lemma suffices, by \eqref{normalisation}.

\begin{lemma} \label{RestrictionEstimate} Let $\phi: \bZ \to \bC$ with $|\phi| \le \nu$, and let $u > 2t$ be a real number. Then
\[
\int_\bT |\hat \phi (\alp)|^u \d \alp \ll_u N^{u-1}.
\]
\end{lemma}

We proceed in stages. The factor of $X^\eps$ in the assumed bound on the number of solutions to \eqref{hyp} is a formidable hurdle, as Lemma \ref{minorbound} fails to provide a power saving on minor arcs. We shall reduce this to a logarithmic factor, at some intermediate exponent $v$; here $v$ will be a real number with
\begin{equation} \label{vrange}
2t < v < u.
\end{equation}
In order to obtain a power saving on minor arcs, we will introduce thicker major arcs, and use them to study the auxiliary majorant
\[
\mu(n) = \sig(b)^{-1} \sum_{\substack{x \in [X]:\\ Wn - b = x^d}} dx^{d-1},
\]
wherein we recall \eqref{sigdef}. The reader should rest assured that the M\"obius function will not appear in this manuscript, and so $\mu$ will always be defined as above. Observe that
\begin{equation*}
\nu(n) \le L \cdot \mu(n) \qquad (n \in \bZ).
\end{equation*}

We shall prove the following restriction estimate for $\mu$, which will serve as a platform from which to attack Lemma \ref{RestrictionEstimate}.

\begin{lemma} \label{platform}
Let $\psi: \bZ \to \bC$ with $|\psi| \le \mu$, and let $v > 2t$ be a real number. Then
\[
\int_\bT |\hat \psi (\alp)|^v \d \alp \ll_v N^{v-1} L^v.
\]
\end{lemma}

Let us explain how this implies Lemma \ref{RestrictionEstimate}, following Bourgain's strategy \cite[\S 4]{Bou1989}. We apply Lemma \ref{platform} with $\psi = L^{-1} \phi$, and with $v$ in the range \eqref{vrange}, obtaining
\begin{equation} \label{inter}
\int_\bT |\hat \phi (\alp)|^v \d \alp \ll N^{v-1} L^{2v}.
\end{equation}
In this section only, we denote by $\del$ an arbitrary parameter in the range
\[
0 < \del < 1,
\]
for consistency with previous literature. This is not to be confused with the density $\del$ defined in \eqref{delta}. 

Consider the large spectra
\[
\cR_\del = \{ \alp \in \bT: |\hat \phi(\alp)| > \del N  \},
\]
which may be regarded as level sets. Note from \eqref{normalisation} that 
\[
\| \hat \phi \|_\infty \le \| \phi \|_1 \le \| \nu \|_1 < (1+\eps)N.
\]
By a standard argument involving dyadic intervals, it suffices to show that if $\eps_0 > 0$ then
\begin{equation} \label{toshow}
\meas (\cR_\del) \ll_{\eps_0}  \frac1 {\del^{v+\eps_0}N}
\end{equation}
(see the discussion surrounding \cite[Lemma 6.3]{BP2016}).

In our quest to establish \eqref{toshow}, we begin by noting that if $\del \le L^{-2v/\eps_0}$ then by \eqref{inter} we have
\[
(\del N)^v \meas(\cR_\del) \le \int_\bT |\hat \phi (\alp)|^v \d \alp \ll N^{v-1} L^{2v},
\]
whereupon
\[
\meas(\cR_\del) \ll \frac{L^{2v}}{\del^v N} \ll \frac1 {\del^{v+\eps_0}N}.
\]
Thus, we may assume that 
\begin{equation} \label{wma}
 L^{-2v/\eps_0} < \del < 1.
\end{equation}

Let $\tet_1, \ldots, \tet_R$ be $N^{-1}$-spaced points in $\cR_\del$. As $v > 2t \ge 2d$, it suffices to show that
\begin{equation} \label{goal}
R \ll_{\eps_0} \del^{-2d-\eps_0}.
\end{equation}
Put
\begin{equation} \label{gamdef}
\gam = d + \eps_0/3.
\end{equation}
Routinely, as in \cite[\S 4]{Bou1989} and \cite[\S 6]{BP2016}, we have
\begin{equation} \label{Bourg}
\del^{2 \gam} N^\gam R^2 \ll \sum_{1 \le r,r' \le R} |\hat \nu (\tet_r - \tet_{r'})|^\gam.
\end{equation}
This calculation may be found in \cite[\S 6]{BP2016}; see also \eqref{Bourg3} and its subsequent derivation.

Recall our Hardy--Littlewood dissection from \S \ref{exponential}. In this dissection, we now specify that $\sig_0$ is large in terms of $\eps_0$ and $u$. Consider 
\[
\tet = \tet_r - \tet_{r'}
\]
in the summand on the right hand side of \eqref{Bourg}. By Lemma \ref{minorbound}, the contribution from $\tet \in \fm$ to the right hand side of \eqref{Bourg} is $O(R^2 N^\gam L^{-\sig_0 \gam})$, with $\sig_0$ large. By \eqref{wma}, this is $o(\del^{2 \gam} N^\gam R^2)$. Hence
\begin{equation} \label{Bourg2}
\del^{2 \gam} N^\gam R^2 \ll \sum_{\substack{1 \le r,r' \le R:\\ \tet \in \fM}} |\hat \nu (\tet_r - \tet_{r'})|^\gam.
\end{equation}

Let $Q = C_5 + \del^{-5}$, with $C_5$ a large positive constant. By \eqref{record}, the contribution to the right hand side of \eqref{Bourg2} from denominators $q > Q$ is bounded, up to a constant, by
\[
R^2 N^\gam (Q^{\eps-\gam/2} + e^{-C_4 \gam \sqrt L}).
\]
This is negligible compared to the left hand side of \eqref{Bourg2}, by \eqref{wma} and the fact that $C_5$ is large. We thus conclude from \eqref{record}, \eqref{wma} and \eqref{Bourg2} that
\[
\del^{2 \gam} R^2 \ll \sum_{q \le Q} \: \sum_{\substack{a \mmod q \\ (a,q)=1}} \: \sum_{1 \le r,r' \le R}
\frac{q^{\eps-\gam/2}}{(1+N | \tet_r - \tet_{r'}- a/q |)^\gam}.
\]
Hence
\begin{equation}\label{BourgainExpression}
\del^{2 \gam} R^2 \ll \sum_{1 \le r,r' \le R} G(\tet_r - \tet_{r'}),
\end{equation}
where 
\[
G(\alp) = \sum_{q \le Q} \: \sum_{a=0}^{q-1}
\frac{q^{\eps - \gam/2}}{(1+N|\sin(\alp-a/q)|)^\gam}.
\]

The inequality \eqref{BourgainExpression} is very similar to \cite[Eq. (4.16)]{Bou1989}, but with $N^2$ replaced by $N$. We have an additional factor of $q^\eps$ in the definition of $G(\alp)$, and we save a $\gamma$th power in the denominator, whereas Bourgain saves only a ($\gamma/2$)nd power. Bourgain's argument carries through, and we obtain \eqref{goal}. 

We have shown that Lemma \ref{platform} implies Lemma \ref{RestrictionEstimate}. To prove Lemma \ref{platform}, we begin by establishing the following `$\eps$-sharp' $(2t)$th moment bound.

\begin{lemma} \label{intermediate} Let $\psi: \bZ \to \bC$ with $|\psi| \le \mu$. Then
\[
\int_\bT |\hat \psi (\alp)|^{2t} \d \alp \ll N^{2t-1+\eps}.
\]
\end{lemma}

\begin{proof} By orthogonality, we have
\begin{align*}
\int_\bT |\hat \psi (\alp)|^{2t} \d \alp &= \int_\bT \sum_{\bn \in \bZ^{2t}} \psi(n_1) \cdots \psi(n_t) \overline{\psi(n_{t+1})} \cdots \overline{\psi(n_{2t})} \\
& \qquad \qquad e(\alp(n_1+ \ldots + n_t - n_{t+1} -  \ldots - n_{2t})) \d \alp \\
&= \sum_{\substack{\bn: \\ n_1 + \ldots + n_t = n_{t+1} + \ldots + n_{2t}}}  \psi(n_1) \cdots \psi(n_t) \overline{\psi(n_{t+1})} \cdots \overline{\psi(n_{2t})} .
\end{align*}
The triangle inequality now gives
\begin{align*}
\int_\bT |\hat \psi (\alp)|^{2t} \d \alp &\le \sum_{\substack{\bn: \\ n_1 + \ldots + n_t = n_{t+1} + \ldots + n_{2t}}}   \mu(n_1) \cdots \mu(n_{2t}) \\
&\ll X^{2t(d-1)}
 \sum_{\substack{\bz \in [X]^{2t}: \\
z_1^d + \ldots + z_t^d = z_{t+1}^d + \ldots + z_{2t}^d}} 1.
\end{align*}
Our hypothesis on $t$ now yields
\[
\int_\bT |\hat \psi (\alp)|^{2t} \d \alp \ll X^{2t(d-1)} X^{2t-d+\eps} = X^{2td-d+\eps}
\]
so, by \eqref{Wbound} and \eqref{Ndef}, we finally have
\[
\int_\bT |\hat \psi (\alp)|^{2t} \d \alp \ll (NL)^{2t-1} X^\eps \ll N^{2t-1+\eps}.
\]
\end{proof}

We now prove Lemma \ref{platform}, again using Bourgain's strategy. We begin by obtaining pointwise estimates for $\hat \mu$. The triangle inequality gives
\begin{align*}
\sig(b) \hat \mu(\tet) &= e(\tet b/W) \sum_{\substack{x \le X:\\ x^d \equiv -b \mmod W}} dx^{d-1} e(\tet x^d / W) \\
&\ll  \Biggl |\sum_{\substack{x \le X:\\ x^d \equiv -b \mmod W}} x^{d-1} e(\tet x^d / W) \Biggr |.
\end{align*}
Hence, by partial summation, we obtain
\begin{equation} \label{mupartial}
\hat \mu(\tet) \ll X^{d/2} + X^{d-1} \sup_{X^{1/2} \le P \le X} |g(\tet;P)|,
\end{equation}
where
\[
g(\tet; P) = \sig(b)^{-1} \sum_{\substack{x \le P:\\ x^d \equiv -b \mmod W}} e(\tet x^d / W).
\]
Moreover, by \eqref{sigdef}, we have
\begin{equation} \label{gsup}
g(\tet; P) \ll \sup_{z \in [W]} \Biggl| \sum_{\substack{x \le P:\\ x \equiv z \mmod W}} e(\tet x^d / W) \Biggr|.
\end{equation}
Writing $x = Wy + z$, we find that
\begin{equation} \label{descend}
\sum_{\substack{x \le P:\\ x \equiv z \mmod W}} e(\tet x^d / W)
= \sum_{y \le P/W} e(W^{d-1} \tet h(y)) + O(1),
\end{equation}
for some monic polynomial $h$ of degree $d$. 

Let $P \in [X^{1/2}, X]$ and $z \in [W]$. The Weyl sums
\[
g_1(\alp) := \sum_{y \le P/W} e(\alp h(y))
\]
are very classical, and are discussed in many texts. For reasons of economy, we employ Baker's estimates \cite{Bak1986}, as packaged in \cite[\S 2]{Cho2016}. The bounds apply to monic polynomials $h$ of degree $d$, and are uniform in the other coefficients of $h$; in particular, they are uniform in $z$. It is plain from the proof of \cite[Lemma 2.3]{Cho2016} that the quantity $\sig(d)$ therein may be replaced by $2^{1-d}$. We conclude thus.
\begin{lemma} \label{wps}
If
\[
|g_1(\alp)| > (P/W)^{1-2^{1-d}+\eps}
\]
then there exist relatively prime integers $r > 0$ and $b$ such that
\[
g_1(\alp) \ll r^{\eps-1/d}PW^{-1} (1+(P/W)^d |\alp - b/r|)^{-1/d}.
\]
\end{lemma}

From Lemma \ref{wps}, we deduce that if
\[
|g_1(W^{d-1}\tet)| > X^{1-2^{1-d}+\eps}
\]
then there exist relatively prime integers $r > 0$ and $b$ such that
\[
g_1(W^{d-1}\tet) \ll r^{\eps - 1/d} XW^{-1} (1+(X/W)^d|W^{d-1}\tet - b/r|)^{-1/d}.
\]
In this case we can put
\[
a = \frac b{(b,W^{d-1})}, \qquad q = \frac{rW^{d-1}}{(b,W^{d-1})},
\]
thus obtaining relatively prime integers $q > 0$ and $a$ such that
\begin{equation} \label{g1bound}
g_1(W^{d-1}\tet) \ll Xq^{\eps-1/d} (1+X^d W^{-1} |\tet - a/q|)^{-1/d}.
\end{equation}
Write
\begin{equation} \label{minor2}
\fn = \{ \tet \in \bT: |\hat \mu (\tet)| \le X^{d-2^{-d}} \}.
\end{equation}
In light of \eqref{Wbound} and \eqref{Ndef}, we can collect \eqref{mupartial}, \eqref{gsup}, \eqref{descend} and \eqref{g1bound} to obtain the following `major arc estimate': if $\tet \in \bT \setminus \fn$ then there exist relatively prime integers $q$ and $a$ such that $0 \le a \le q-1$ and
\begin{equation} \label{major2}
\hat \mu (\tet) \ll NLq^{\eps-1/d}(1+N | \tet - a/q |)^{-1/d}.
\end{equation}

Now that we have made the necessary preparations, we complete the proof of Lemma \ref{platform}. This will parallel our proof that Lemma \ref{RestrictionEstimate} follows from Lemma \ref{platform}. Consider the large spectra
\[
\cR_\del = \{ \alp \in \bT: |\hat \psi(\alp)| > \del NL  \},
\]
noting from \eqref{Wbound} and \eqref{Ndef} the crude bound
\begin{equation} \label{mu1}
\| \hat \psi \|_\infty \le \| \psi \|_1 \le \| \mu \|_1 \le \sum_{x \le X} dx^{d-1} \sim X^d \le NL.
\end{equation}
Similarly to before, it suffices to show that if $\eps_0 > 0$ then we have 
\[
\meas (\cR_\del) \ll_{\eps_0}  \frac1 {\del^{2t+\eps_0}N}.
\]
This time, we can use Lemma \ref{intermediate} to reduce consideration to $\del$ in the range
\begin{equation} \label{wma2}
N^{-\eps} < \del < 1,
\end{equation}
wherein we recall our notational convention for $\eps$.

With $\tet_1, \ldots, \tet_R$ be as $N^{-1}$-spaced points in $\cR_\del$, it remains to show \eqref{goal}. Again with \eqref{gamdef}, we will find that
\begin{equation} \label{Bourg3}
\del^{2\gam} N^\gam L^{\gam} R^2 \ll \sum_{1 \le r,r' \le R} |\hat \mu (\tet_r - \tet_{r'})|^\gam.
\end{equation}
We now verify this inequality by following the corresponding argument in the proof of \cite[Lemma 6.3]{BP2016}.

Let $a_n \in \bC$ be such that $|a_n| \le 1$ and $\psi(n) = a_n \mu(n)$, for $n \in [N]$. Furthermore, let $c_1, \ldots, c_R \in \bC$ be such that $|c_r| = 1$ and
\[
c_r \hat \psi(\tet_r) = |\hat \psi(\tet_r)| \qquad(1 \le r \le R).
\]
It follows from the Cauchy--Schwarz inequality and \eqref{mu1} that
\begin{align*}
\del^2 N^2 L^2 R^2
&\le \Biggl( \sum_{r \in [R]} |\hat \psi(\tet_r)| \Biggr)^2 
= \Biggl( \sum_{r \in [R]} c_r \sum_n a_n \mu(n) e(n \tet_r) \Biggr)^2 \\
&\le \| \mu\|_1 \sum_n \mu(n) \Biggl | \sum_{r \in [R]} c_r e(n\tet_r) \Biggl |^2
\ll NL \sum_n \mu(n) \Biggl | \sum_{r \in [R]} c_r e(n\tet_r) \Biggl |^2,
\end{align*}
and so
\[
\del^2 NL R^2 \ll \sum_{1 \le r, r' \le R} |\hat \mu(\tet_r - \tet_{r'})|.
\]
An application of H\"older's inequality now harvests \eqref{Bourg3}.

Consider $\tet = \tet_r - \tet_{r'}$ in the summand on the right hand side of \eqref{Bourg3}. By \eqref{Wbound}, \eqref{Ndef} and \eqref{minor2}, the contribution from $\tet \in \fn$ to the right hand side of \eqref{Bourg2} is $O(R^2 N^{\gam(1 + \eps - 2^{-d}/d)})$. By \eqref{wma2}, this is $o(\del^{2 \gam} N^\gam L^{\gam} R^2)$. Hence
\begin{equation} \label{Bourg4}
\del^{2 \gam} N^\gam L^\gam R^2 \ll \sum_{\substack{1 \le r,r' \le R:\\ \tet \notin \fn}} |\hat \mu (\tet_r - \tet_{r'})|^\gam.
\end{equation}

Let $Q = C_6 + \del^{-3d}$, with $C_6$ a large positive constant. By \eqref{major2}, the contribution to the right hand side of \eqref{Bourg4} from denominators $q > Q$ is bounded, up to a constant, by
\[
R^2 N^\gam L^{\gam} Q^{\eps-\gam/d}.
\]
This is negligible compared to the left hand side of \eqref{Bourg4}, as $C_6$ is large. We thus conclude from \eqref{major2} and \eqref{Bourg4} that
\[
\del^{2 \gam} R^2 \ll \sum_{q \le Q} \: \sum_{\substack{a \mmod q \\ (a,q)=1}} \: \sum_{1 \le r,r' \le R}
\frac{q^{\eps-\gam/d}}{(1+N | \tet_r - \tet_{r'}-a/q |)^{\gam/d}}.
\]
Hence
\begin{equation}\label{BourgainExpression2}
\del^{2 \gam} R^2 \ll \sum_{1 \le r,r' \le R} G_2(\tet_r - \tet_{r'}),
\end{equation}
where 
\[
G_2(\alp) = \sum_{q \le Q} \: \sum_{a=0}^{q-1}
\frac{q^{\eps- \gam/d}}{(1+N|\sin(\alp-a/q)|)^{\gam/d}}.
\]

The inequality \eqref{BourgainExpression2} is very similar to \cite[Eq. (4.16)]{Bou1989}, but with $N^2$ replaced by $N$, and with $Q \sim \del^{-3d}$ rather than $Q \sim \del^{-5}$. The exponents differ but, since $\gam > d$, Bourgain's argument carries through, and provides the desired bound \eqref{goal} for $R$. This completes the proof of Lemma \ref{platform}. We have established all of the results in this section. In particular, we know from Lemma \ref{RestrictionEstimate} that $\nu$ satisfies a restriction estimate at any exponent $u > 2t$.

\section{The $K$-trivial count}
\label{trivial}

In this section we show that $\nu$ saves $1/s$ on $K$-trivial solutions. Let $t \in \bN$ be such that the number of solutions $\bz \in [X]^{2t}$ to \eqref{hyp} is $O_{t,d,\eps}(X^{2t-d+\eps})$, and assume $s > 2t$. 

\begin{lemma} \label{trivialstep}
The number of $\bx \in [X]^s$ with $(x_1^d, \ldots, x_s^d) \in K$ is
\[
O_{k,s,d,\eps}(X^{s-d-d/(s-1)+\eps}).
\]
\end{lemma}

\begin{proof} The set $K$ lies in the union of $k$ subspaces of the form
\[
\{ \by \in \bQ^s: \bc \cdot \by = \bd \cdot \by = 0 \},
\]
where $\bd \in \bQ^s$ is a fixed vector that is not proportional to $\bc$. Our task, therefore, is to count solutions  $\bx \in [X]^s$ to the system
\begin{equation} \label{system}
c_1 x_1^d + \ldots + c_s x_s^d = d_1 x_1^d + \ldots + d_s x_s^d = 0.
\end{equation}
From \eqref{system} we obtain
\[
e_1 x_1^d + \ldots + e_{s-1} x_{s-1}^d = 0,
\]
where $(e_1, \ldots, e_{s-1}) \ne \bzero,$ and by rescaling we may assume that
\[
(e_1, \ldots, e_{s-1}) \in \bZ^{s-1}.
\]
Let $u$ be the number of nonzero $e_i$, and note that $1 \le u \le s-1$. Without loss of generality $e_1, \ldots, e_u \in \bZ \setminus \{ 0 \}$ and $e_{u+1} = \ldots = e_{s-1} = 0$, so that
\begin{equation} \label{reduced}
e_1 x_1^d + \ldots + e_u x_u^d = 0.
\end{equation}

By orthogonality, the number $\cM$ of solutions $(x_1, \ldots, x_u) \in [X]^u$ to \eqref{reduced} is
\[
\int_\bT \sum_{\by \in [X]^u} e\Bigl(\alp \sum_{i \le u} e_i y_i^d \Bigr) \d \alp 
= \int_\bT \Bigl(\prod_{i \le u}\sum_{y_i \le X} e(\alp e_i y_i^d)\Bigr) \d \alp.
\]
By H\"older's inequality we now have, for some $i \in [u]$,
\[
\cM \ll \int_\bT |f(e_i \alp)|^u \d \alp,
\]
where 
\[
f(\tet) = \sum_{x \le X} e(\tet x^d).
\]
Note that $e_i \ne 0$, since $i \in [u]$. By periodicity, a change of variables reveals that
\[
\int_\bT |f(e_i \alp)|^u\d \alp = \int_\bT |f(\alp)|^u \d \alp.
\]
Another application of H\"older's inequality now gives
\[
\cM \ll \int_\bT |f(\alp)|^u \d \alp \ll \Bigl(\int_\bT |f(\alp)|^{s-1} \d \alp \Bigr)^{u/(s-1)}.
\]
As $s-1 \ge 2t$ and $|f(\alp)| \le X$, we now have
\begin{equation} \label{Mbound}
\cM \ll \Bigl(X^{s-1-2t} \int_\bT |f(\alp)|^{2t} \d \alp \Bigr)^{u/(s-1)}.
\end{equation}

From \eqref{system} we have
\begin{equation} \label{determine}
c_{u+1} x_{u+1}^d + \ldots + c_s x_s^d = -(c_1 x_1^d + \ldots + c_u x_u^d).
\end{equation}
Given integers $x_1, \ldots, x_u$, the number of solutions $(x_{u+1}, \ldots, x_s) \in \bZ^{s-u}$ to \eqref{determine} is, by orthogonality, at most
\[
\int_\bT \Biggl| \sum_{\bz \in [X]^{s-u}} e\Bigl(\alp \sum_{i \le s-u} c_{u+i} z_i^d \Bigr) \Biggr| \d \alp.
\]
By following our calculation bounding $\cM$, we deduce that this quantity is bounded by
\[
\Bigl(X^{s-1-2t} \int_\bT |f(\alp)|^{2t} \d \alp \Bigr)^{(s-u)/(s-1)}.
\]
Coupling this information with \eqref{Mbound}, we find that the number $\cN$ of solutions $\bx \in [X]^s$ to \eqref{system} satisfies
\[
\cN \ll \Bigl(X^{s-1-2t} \int_\bT |f(\alp)|^{2t} \d \alp \Bigr)^{s/(s-1)}.
\]
By orthogonality, the integral $\int_\bT |f(\alp)|^{2t} \d \alp$ equals the number of solutions $\bz \in [X]^{2t}$ to \eqref{hyp} which, by hypothesis, is $O(X^{2t-d+\eps})$. Hence
\[
\cN \ll (X^{s-1-2t} X^{2t-d+\eps})^{s/(s-1)} \ll X^{(s-1-d)s/(s-1)+2\eps} = X^{s-d - d/(s-1) + 2 \eps},
\]
which proves the lemma.
\end{proof}

\begin{cor} The majorant $\nu$ saves $1/s$ on $K$-trivial solutions.
\end{cor}

\begin{proof} By \eqref{normalisation}, our task is to establish the inequality
\[
\sum_{\by \in K} \prod_{i=1}^s \nu(y_i) \ll_{k,s,d} N^{s-1-1/s}.
\]
By \eqref{nudef} and \eqref{supremum}, we have
\[
\sum_{\by \in K} \prod_{i=1}^s \nu(y_i) \ll (X^{d-1}L)^s \sum_{\substack{\bx \in [X]^s:\\ (x_1^d + b, \ldots, x_s^d + b) / W \in K}} 1. 
\]
By the definition of $K$, the condition $(x_1^d + b, \ldots, x_s^d + b) / W \in K$ is equivalent to the condition
\[
 (x_1^d, \ldots, x_s^d) \in K.
\]
Now Lemma \ref{trivialstep} yields
\[
\sum_{\by \in K} \prod_{i=1}^s \nu(y_i) \ll (X^{d-1}L)^s X^{s-d-d/(s-1)+\eps} \ll X^{d(s-1)-d/(s-1)+2\eps},
\]
so by \eqref{Wbound} and \eqref{Ndef} we have
\[
\sum_{\by \in K} \prod_{i=1}^s \nu(y_i) \ll N^{s-1-1/(s-1)+\eps} \ll N^{s-1-1/s}.
\]
\end{proof}

\section{The density bound}
\label{density}

Finally, we have all of the ingredients needed to prove Theorem \ref{MainThm}. Recall \eqref{fixb}. Note that $\cA$ has only $K$-trivial solutions to \eqref{LinearEquation}, in the sense that if $\bn \in \cA^s$ and $\bc \cdot \bn = 0$ then $\bn \in K$. Indeed, suppose $n_1, \ldots, n_s \in \cA$ and $\bc \cdot \bn = 0$. Then, by \eqref{zerosum} and \eqref{cAdef}, we have \eqref{ActualEquation} with
\[
x_i = (Wn_i - b)^{1/d} \in A \qquad (1 \le i \le s).
\]
Our hypothesis on $A$ then tells us that 
\[
(Wn_1 -b, \ldots, Wn_s - b) = (x_1^d, \ldots, x_s^d) \in K.
\]
From its construction, we see that $K$ is invariant under translations and dilations, so we now have $\bn \in K$, which confirms that $\cA$ has only $K$-trivial solutions to \eqref{LinearEquation}.

We apply \cite[Proposition 2.8]{BP2016} to the majorant $\nu$, noting that $\cA \subseteq \supp(\nu)$. We showed in \S \ref{decay} that $\nu$ has Fourier decay of level $O(w^{\eps-1/2})$. We showed in \S \ref{Restriction} that $\nu$ satisfies a restriction estimate at the exponent $s-1/2$. We showed in \S \ref{trivial} that $\nu$ saves $1/s$ on trivial solutions. Hence
\[
\sum_{n \in \cA} \nu(n) \ll \frac N{ \min \{ \log \log (w^{1/2-\eps}), \log N \}^{s-2-\eps}},
\]
so by \eqref{Wdef} we have
\[
\sum_{n \in \cA} \nu(n) \ll N(\log \log \log \log X)^{2+\eps-s}.
\]
Coupling this with \eqref{DensityTransfer} yields
\[
\del^d N \ll  N(\log \log \log \log X)^{2+\eps-s},
\]
and so $\del \ll (\log \log \log \log X)^{\frac{2-s}d + \eps}$. By \eqref{delta}, this gives \eqref{DensityBound}, completing the proof of Theorem \ref{MainThm}.

\bibliographystyle{amsbracket}
\providecommand{\bysame}{\leavevmode\hbox to3em{\hrulefill}\thinspace}

\end{document}